\documentclass[11pt,a4paper]{amsart}
\usepackage[utf8]{inputenc}
\usepackage{amsmath}
\usepackage{amsfonts}
\usepackage{amssymb}
\usepackage{amsthm}
\usepackage[all]{xy}

\newtheorem*{theorem}{Theorem}
\newtheorem{lemma}{Lemma}

\newcommand{\Z}{\mathbb{Z}}

\DeclareMathOperator{\im}{im}

\author{Christopher Wulff}

\address{Instituto de Matem\'aticas (Unidad Cuernavaca), Universidad Nacional Aut\'o\-noma de M\'exico,
Avenida Universidad s/n, Colonia Lomas de Chamilpa, 62210 Cuernavaca, Morelos, Mexico}
\email{christopher.wulff@im.unam.mx}

\title[First $L^2$-Betti number of a fibration]{A geometric proof of Lück's vanishing theorem for the first $L^2$-Betti number of the total space of a fibration}

\thanks{Supported by the Program of Post-Doctoral
Scholarships at the Universidad Nacional Aut\'onoma de M\'exico.}

\begin{document}

\begin{abstract}
A significant theorem of Lück says that the first $L^2$-Betti number of the total space of a fibration vanishes under some conditions on the fundamental groups.
The proof is based on constructions on chain complexes.
In the present paper, we translate the proof into the world of CW-complexes to make it more accessible.
\end{abstract}

\maketitle

\section{Introduction}
In \cite{Lueck}, Lück proved the following significant theorem:

\begin{theorem}[{\cite[Theorem 3.1]{Lueck}}]
Let $F\xrightarrow{i} E\xrightarrow{p}B$ be a fibration of connected CW-complexes such that $F$ and $B$ have finite $2$-skeletons. Then $E$ has finite $2$-skeleton up to homotopy. If the image of $\pi_1(F)\to\pi_1(E)$ is infinite and $\pi_1(B)$ contains $\Z$ as a subgroup, then the first $L^2$-Betti number of $E$ vanishes:
$b_1(E)=0$.
\end{theorem}

Lück's proof is based on somewhat abstract constructions in the world of chain complexes, which make it quite hard to understand what really is going on geometrically.

On closer examination, however, it turns out that most of these constructions do have a counterpart already at the level of CW-complexes.
The purpose of the present paper is to elaborate these geometric counterparts and thereby translating Lück's proof into the world of CW-complexes.
The hope is that the geometric version of the proof is more accessible to the generic reader.

It should be said that the present paper is not meant to be considered independently of the original paper \cite{Lueck}.
In particular, we use without any recapitulation the same notation and assume that the reader is familiar with the basic results in \cite[Sections 1 \& 2]{Lueck}.

Furthermore, we only re-prove the original theorem shown above, but not any generalization such as \cite[Theorem 6.67]{LueckBuch} (although the proof of the latter theorem contains a geometric construction which exhibits a slight similarity to what we do here).
After all, the purpose of this paper is to simplify matters, not to complicate them.

The fact that our proof takes more space then the original proof in \cite{Lueck} is mainly due to the fact that we included a few more details.

\section{Outline of proof}

The idea of the proof is as follows. We will construct a more accessible CW-complex $T$ and a $1$-connected map $h:T\to E$ for which we can directly prove 
\[b_1(E)\leq b_1(T,(h_*)^*\ell^2(\pi_1(E)))=0\,.\]

In preparation of the proof we shall make the set-up precise. First of all, we assume that $E$ has finite $2$-skeleton and we can and will assume that all maps appearing (including loops defined on the unit interval $I:=[0,1]$ with the obvious cell structure) are cellular. And secondly, we choose $0$-cells $e\in E$ and $b:=p(e)\in B$ as basepoints and let $F:=p^{-1}(b)$ with basepoint $e\in F$.

Now, denote $\pi:=\pi_1(B,b)$, $\Gamma:=\pi_1(E,e)$ and $\Delta=\im(i_*:\pi_1(F,e)\to\Gamma)$. Thus, we obtain a group extension
\[1\to\Delta\to\Gamma\xrightarrow{p_*}\pi\to 1\,.\]

For each $w\in\pi$ we chose some arbitrary pre-image $\overline{w}\in\Gamma$ under $p_*$. We shall also use the same letters $w,\overline{w}$ for representing loops  $I\to B$, $I\to E$, respectively, and assume $w=p\circ \overline{w}$.

Choose a solution $h(w)$ to the lifting problem
\[\xymatrix@C=16ex{
F\times\{0\}\cup\{e\}\times I \ar[r]^-{i\cup\overline{w}} \ar[d]_{\text{incl.}}
&E\ar[d]^{p}
\\F\times I\ar[r]_-{w\circ \operatorname{pr}_I}\ar[ur]^{h(w)}
&B
}\]
and denote $\sigma(w):=h(\_,1):(F,e)\to (F,e)$.
The pointed homotopy class of $\sigma(w)$ is independent of the choices made and called the pointed fibre transport along $w$.
Denote by 
\[T_{\sigma(w)}:=F\times I/(x,1)\sim(\sigma(w)x,0)\]
the mapping torus of $\sigma(w)$.

We are now ready to define the CW-complex $T$. Choose a generating set $S=\{s_1,\dots, s_g\}$ of $\pi$ such that $s_1$ has infinite order and apply the constructions above to each $w\in S$. Then $T$ is obtained by gluing together $T_{\sigma(s_1)},\dots,T_{\sigma(s_g)}$ along the common subcomplex $F\times\{0\}$. It is obviously connected, because $F$ is connected.

All $h(s_1),\dots,h(s_g)$ together assemble to a map $h:T\to E$ which fits into the self-explaining commutative diagram
\[\xymatrix@C=16ex{
F\ar@{=}[d]\ar[r]^{\times\{0\}}
&T\ar[d]^h\ar[r]^{\text{proj.}}
&\bigvee_{n=1}^gS^1\ar[d]^{\bigvee_{n=1}^gs_n}
\\F\ar[r]^i
&E\ar[r]^p&B\,.
}\]
On fundamental groups, this induces
\[\xymatrix@C=16ex{
\pi_1(F,e)\ar@{->>}[d]^{i_*}\ar[r]
&\pi_1(T,(e,0))\ar[d]^{h_*}\ar@{->>}[r]
&\Z^{*g}\ar@{->>}[d]
\\\Delta\ar@{^{(}->}[r]
&\Gamma\ar@{->>}[r]^{p_*}&\pi
}\]
and exactness of the lower row together with the indicated surjectivity of some of the maps immediately implies that $h_*$ is surjective, too, and so $h$ is $1$-connected.

Denote by $\widetilde E\to E$ and $\widetilde T\to T$ the universal coverings and by $\widehat{T}\xrightarrow{t} T$ the connected covering of $T$ associated to the subgroup $\ker(h_*)$. Thus, the latter has deck transformation group $\Gamma$ and there is a $\Gamma$-equivariant lift $\widehat{h}:\widehat{T}\to \widetilde E$ of $h$.
We obtain a $1$-connected $\Z\Gamma$-chain map of free $\Z\Gamma$-chain complexes
\[\Z\Gamma\otimes_{\Z\Gamma}C_*(\widetilde T)\cong C_*(\widehat{T})\xrightarrow{\widehat{h}_*}C_*(\widetilde E)\]
and the proof of \cite[Lemma 1.2.1]{Lueck} implies
\[b_1(E)\leq b_1(T,(h_*)^*\ell^2\Gamma)\,.\]

In the following section, we shall provide a more concrete construction of $\widehat{T}$ which allows us to calculate the right hand side of this inequality directly in the final section.

\section{Explicit construction of $\widehat{T}$}

Denote by $f:\overline{F}\to F$ the connected covering corresponding to the subgroup $\ker(i_*)$, which has $\Delta$ as deck transformation group. 
Choose any $0$-cell $\overline{e}\in\overline{F}$ with $f(\overline{e})=e$ as basepoint.

For arbitrary $w$, the map $h(w):F\times I\to E$ is a homotopy between $i\circ\sigma(w)$ and $i$, which implies
\[\im((\sigma(w)\circ f)_*)=(\sigma(w))_*(\ker(i_*))=\ker(i_*)=\im(f_*)\]
and thus $\sigma(w)$ lifts to a map $\overline{\sigma}(w):\overline{F}\to\overline{F}$ which fixes $\overline{e}$.
This map is not $\Delta$-equivariant, but:
\begin{lemma}\label{lem:noncommuting}
For arbitrary $\delta\in\Delta$ we have
\[\overline{\sigma}(w)\circ\delta= \underbrace{\overline{w}^{-1}\delta\overline{w} }_{\in\Delta}\,\circ\,\overline{\sigma}(w)\,.\]
\end{lemma}
\begin{proof}
Note that both sides are lifts $\overline{F}\to\overline{F}$ of the map $\sigma(w):F\to F$. It therefore suffices to prove the equality at the point $\overline{e}$, i.\,e.\ that $\overline{\sigma}(w)(\delta\cdot\overline{e})= (\overline{w}^{-1}\delta\overline{w})\cdot \overline{e}$.

Denote a representative loop $I\to F\subset E$ of $\delta$ by the same letter and let $\overline{\delta}$ be a lift of $\delta$ to $\overline{F}$ with $\overline{\delta}(0)=\overline{e}$. Then $\delta\cdot \overline{e}$ is defined as $\overline{\delta}(1)$.

With this data at hand, the point $\overline{\sigma}(w)(\delta\cdot\overline{e})$ is defined as $\alpha(1)$,  where $\alpha:I\to \overline{F}$ is the lift of the loop $\sigma(w)\circ\delta$ with starting point $\alpha(0)=\overline{e}$. In other words, the action of $\sigma(w)\circ\delta\in\Delta$ takes $\overline{e}$ to $\overline{\sigma}(w)(\delta\cdot\overline{e})$.

But $\sigma(w)\circ\delta= \overline{w}^{-1}\delta\overline{w}$ in $\Gamma$ and therefore also in $\Delta$, because $h(w)$ gives rise to a homotopy in $E$ between those loops. This proves the claim.
\end{proof}

Denote by $\widehat{F}:=\Gamma\times\overline{F}/\Delta$ the $\Gamma$-CW-complex obtained from $\Gamma\times\overline{F}$ by dividing out the equivalence relation $(\gamma,x)\sim(\gamma\delta^{-1},\delta x)$. The $\Gamma$-action is the obvious left action on the first component.
Lemma \ref{lem:noncommuting} now implies, that the $\Gamma$-equivariant map
\[\widehat{\sigma}(w):[(\gamma,x)]\mapsto [(\gamma\overline{w},\overline{\sigma}(w)x)]\]
is well-defined.

Denote by $T_{\widehat{\sigma}(w)}=\widehat{F}\times I/\sim$ the mapping torus of $\widehat{\sigma}(w)$. 
In this section, we define $\widehat{T}$
 by gluing together the $T_{\widehat{\sigma}(s_1)},\dots,T_{\widehat{\sigma}(s_g)}$ along the common $\widehat{F}\times\{0\}$ and claim that it is exactly the covering described in the previous section.

First of all, note that $\widehat{T}$ is indeed a covering of $T$ with each of the subcomplexes $T_{\widehat{\sigma}(s_n)}$ covering the corresponding subcomplex  $T_{\sigma(s_n)}$, and clearly, the canonical $\Gamma$-action on $\widehat{T}$ coming from the action on $\widehat{F}$ is by deck transformations.

\begin{lemma}
The space $\widehat{T}$ is connected.
\end{lemma}
\begin{proof}
Note that it clearly suffices to show that the points of 
\[t^{-1}\{(e,0)\}=\{[(\gamma,\overline{e},0)]\,|\,\gamma\in\Gamma\}\]
can be connected by paths in $\widehat{T}$.

For each $n=1,\dots,g$ and $\gamma\in\Gamma$, the path
\[\widehat{s}_{n,\gamma}:\,I\to T_{\widehat{\sigma}(s_n)}\subset \widehat{T}\,,\quad r\mapsto [(\gamma,\overline{e},r)]\]
connects 
$[(\gamma,\overline{e},0)]$ with \[[(\gamma,\overline{e},1)]=[(\gamma \overline{s_n},\overline{\sigma}(s_n)\overline{e},0)]=[(\gamma \overline{s_n},\overline{e},0)]\]
and is mapped to $\overline{s_n}$, $s_n$ under $h\circ t$ and $p\circ h\circ t$, respectively.

By applying this repeatedly, we see that each $[(\gamma,\overline{e},0)]$ is connected to $[(\delta,\overline{e},0)]=[(1,\delta\overline{e},0)]$ for some $\delta\in \ker(p_*)=\Delta$, and this is in turn is connected to $[(1,\overline{e},0)]$, because $\overline{F}$ is path connected.
\end{proof}

\begin{lemma}
If $\tau:I\to\widehat{T}$ is a path connecting $[(\gamma',\overline{e},0)]$ to 
$[(\gamma'\gamma,\overline{e},0)]$, then 
$h\circ t$ maps $\tau$ to a representative loop $I\to E$ of $\gamma$.
\end{lemma}
\begin{proof}
We have already seen this for $\tau$ being one of the paths $\widehat{s}_{n,\gamma'}$ defined in the proof of the previous lemma. It is also clear for $\tau$ a path within $\widehat{F}\times\{0\}$, because any such path is of the form $r\mapsto [(\gamma',\tau'(r),0)]$ with $\tau'$ a path in $\overline{F}$ satisfying $\tau'(0)=\overline{e}$ and $(\gamma'\gamma,\overline{e})\sim (\gamma',\tau'(1))$, which implies $\gamma\in\Delta$ and $\tau'(1)=\gamma\overline{e}$.

The set of all paths which satisfy the claim is clearly closed under concatenation and taking reversed paths. It is thus sufficient to show that any path $\tau$ satisfying the prerequisites of the lemma can be homotoped into a concatenation of the $\widehat{s}_{n,\gamma'}$ and their inverses and paths within $\widehat{F}\times\{0\}$.

By cellular approximation and a subsequent homotopy within the parameter space $I$, any such $\tau$ can be written as a concatenation of finitely many paths $\tau_1,\dots,\tau_k$, each of which is a constant speed path along a $1$-cell of $\widehat{T}$.
These are either contained in $\widehat{F}\times\{0\}$ or run along a $1$-cell of the form $c\times I\subset T_{\widehat{\sigma}(s_n)}\subset \widehat{T}$ 
with $c$ being a $0$-cell of $\widehat{F}$.
Denote the paths along the latter in positive direction by $\rho_{n,c}$. 
Note that for $c=[(\gamma,\overline{e})]$ we recover the path $\widehat{s}_{n,\gamma}$.

Any $c$ which is not of this form can be connected to some $[(\gamma,\overline{e})]$ by a path $\alpha:I\to\widehat{F}$ and an obvious homotopy in $\widehat{T}$ shows
\[\rho_{n,c}\cdot(\sigma(s_n)\circ\alpha)\simeq \alpha\cdot\widehat{s}_{n,\gamma}\,.\]

This allows us to trade any of the $\tau_m$ which is equal to some $\rho_{n,c}$ (or its inverse) for a concatenation of two paths in $\widehat{F}$ and one of the $\widehat{s}_{n,\gamma}$ (or its inverse) in between.
This shows the claim.
\end{proof}

The last two Lemmas imply that $\widehat{T}$ is exactly the covering associated to $\ker(h_*)$: it is connected and if $\tau:I\to\widehat{T}$ maps to a loop in $T$ based at $e$, then $\tau$ itself is a loop if and only if $h_*[t\circ\tau]=0$.

Furthermore, the last lemma shows that the two canonical actions of $\Gamma$ on $\widehat{T}$ as deck transformations, the action coming from general covering theory and the action induced by the $\Gamma$-action on $\widehat{F}$, are in fact the same.

\section{Calculating $b_1(T,(h_*)^*\ell^2\Gamma)=0$}

The proof of the theorem is completed by calculating $b_1(T,(h_*)^*\ell^2\Gamma)=0$.

Note that $\widehat{T}\setminus T_{\widehat{\sigma}(s_1)}=\coprod_{n=2}^g\widehat{F}\times(0,1)$ and we therefore obtain a short exact sequence of $\Gamma$-chain-complexes
\[0\to C_*(T_{\widehat{\sigma}(s_1)})\to C_*(\widehat{T})\to \bigoplus_{n=2}^g C_{*-1}(\widehat{F})\to 0\,.\]
This induces by \cite[Thm. 2.1 on p.10]{CheegerGromov} a weakly exact $L^2$-homology sequence
\begin{align*}
H_1(\ell^2\Gamma\otimes_{\Z\Gamma}C_*(T_{\widehat{\sigma}(s_1)}))
&\to H_1(\ell^2\Gamma\otimes_{\Z\Gamma}C_*(\widehat{T}))
\to \bigoplus_{n=2}^g H_0(C_*(\ell^2\Gamma\otimes_{\Z\Gamma}\widehat{F}))\,.
\end{align*}

On the right hand side, the van Neumann dimension of the summands is $b_0(F,(i_*)^*\ell^2\Gamma)$, which vanishes by \cite[Lemma 1.2.5]{Lueck} as $\im(i_*)=\Delta$ is infinite.

The van Neumann dimension of the left hand side is  $b_1(T_{\sigma(s_1)},(\phi_*)^*\ell^2\Gamma)$, where $\phi:T_{\sigma(s_1)}=F\times I/{\sim}\to E$ is a quotient of $h(s_1)$.
Let $\Gamma'\subset\Gamma$ be the image of $\phi_*$, which is exactly the subgroup of $\Gamma$ generated by $\Delta$ and $\overline{s_1}$. As $s_1\in B$ has infinite order, the canonical map $\pi_1(T_{\sigma(s_1)},e)\to\Z$ factors as $\pi_1(T_{\sigma(s_1)},e)\xrightarrow{\phi'}\Gamma'\to\Z$.
Using \cite[Lemma 1.2.3 and Theorem 2.1]{Lueck} we conclude 
\[b_1(T_{\sigma(s_1)},(\phi_*)^*\ell^2\Gamma)=b_1(T_{\sigma(s_1)},(\phi')^*\ell^2\Gamma')=0\,.\]

Thus, the weakly exact sequence implies that the van Neumann dimension of the middle term $H_1(\ell^2\Gamma\otimes_{\Z\Gamma}C_*(\widehat{T}))$, which is exactly $b_1(T,(h_*)^*\ell^2\Gamma)$, vanishes as well and the proof of the theorem is complete.

\bibliographystyle{plain}

\begin{thebibliography}{1}

\bibitem{CheegerGromov}
Jeff Cheeger and Mikhael Gromov.
\newblock Bounds on the von {N}eumann dimension of {$L^2$}-cohomology and the
  {G}auss-{B}onnet theorem for open manifolds.
\newblock {\em J. Differential Geom.}, 21(1):1--34, 1985.

\bibitem{Lueck}
Wolfgang L{\"u}ck.
\newblock {$L^2$}-{B}etti numbers of mapping tori and groups.
\newblock {\em Topology}, 33(2):203--214, 1994.

\bibitem{LueckBuch}
Wolfgang L{\"u}ck.
\newblock {\em {$L^2$}-invariants: theory and applications to geometry and
  {$K$}-theory}, volume~44 of {\em Ergebnisse der Mathematik und ihrer
  Grenzgebiete. 3. Folge. A Series of Modern Surveys in Mathematics [Results in
  Mathematics and Related Areas. 3rd Series. A Series of Modern Surveys in
  Mathematics]}.
\newblock Springer-Verlag, Berlin, 2002.

\end{thebibliography}

\end{document}